\documentclass[a4paper,11pt,twoside]{article}

\usepackage{amssymb,amsmath,amsfonts, amsthm,amscd}
\usepackage{multirow}

\usepackage{graphicx}
\usepackage{pstricks,pst-node,pst-text,pst-3d}
\usepackage{psfrag}
\usepackage{verbatim}

\input xy
\xyoption{all}

\newtheorem{theorem}{Theorem}[section]
\newtheorem{lemma}[theorem]{Lemma}
\newtheorem{proposition}[theorem]{Proposition}
\newtheorem{corollary}[theorem]{Corollary}

\newtheorem*{question}{Question}

\theoremstyle{definition}

\newtheorem*{definition}{Definition}

\newcommand{\tor}{\operatorname{tor}}
\newcommand{\Id}{\operatorname{Id}}

\title{Mutation and $\operatorname{SL}(2,\mathbf C)$-Reidemeister torsion for hyperbolic knots.}
\author{Pere Menal-Ferrer \and Joan Porti 
\thanks{Both authors partially supported by the 
Spanish Micinn through grant MTM2009-0759 and by 
the Catalan AGAUR through grant SGR2009-1207. 
The second author received the prize “ICREA Acad\`emia” 
for excellence in research, funded by the Generalitat de Catalunya.}
}

\begin{document}
\maketitle

\begin{abstract}
Given a hyperbolic knot, we prove that the Reidemeister torsion of any lift of the holonomy to  $\operatorname{SL}(2,\mathbf C)$
is invariant under mutation along a Conway sphere. 
\end{abstract}


\section{Introduction}

Let $K\subset S^3$ be a hyperbolic knot and $C\subset S^3$ a Conway sphere. Namely $C$ intersects transversally $K$ in four points.
We write $\tau=\tau_i:(C, C\cap K)\to (C, C\cap K)$ to denote any of the three involutions in Figure~\ref{fig:involutions}, for $i=1,2,3$.

\begin{figure}[t]
\begin{center}
{\input{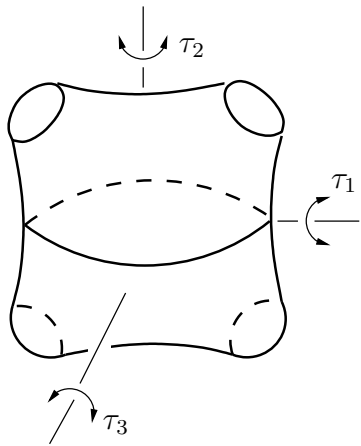}}
\end{center}
   \caption{The involutions on the Conway sphere}\label{fig:involutions}
\end{figure}

The knot $K^\tau\subset S^3$ obtained by cutting along $C$ and gluing again after composing with $\tau$ is called the mutant knot.
We are interested in comparing invariants of $K$ and $K^\tau$, thus we may assume that, if $M= S^3\setminus\mathcal N (K)$ denotes the knot
exterior, then the  surface $S=C\cap M$ is essential in $M$.

Ruberman \cite{Ruberman} showed that $K^\tau$ is also hyperbolic, and that $M^\tau=S^3\setminus\mathcal N (K^\tau)$ has the same volume as $M=S^3\setminus\mathcal N (K)$.
See \cite{DGST,MortonRyder} for a recent account on invariants that distinguish or not $K$ from $K^\tau$.

Let $\rho:\pi_1(M)\to \operatorname{SL}(2,\mathbf C)$ be a lift of the holonomy. If $\mu\in\pi_1(M)$ is a meridian, then $\operatorname{trace}(\rho(\mu))=\pm2 $,
and there are precisely two lifts of the holonomy up to conjugation, one with $\operatorname{trace}(\rho(\mu))=+2 $ and another with 
$\operatorname{trace}(\rho(\mu))= - 2 $.

 By \cite{MenalPorti}, $\rho$ is acyclic, namely the homology and cohomology of $M$ with coefficients twisted by $\rho$
vanish, hence the Reidemeister torsion $\operatorname{tor}(M,\rho)$ is
well defined. Moreover as the dimension of $\mathbf C^2$ is even, there is no sign indeterminacy, thus 
$\operatorname{tor}(M,\rho)$ is a well defined nonzero complex number, independent of the conjugacy class of $\rho$. 
Hence these torsions  are two topological invariants of the hyperbolic knot.

\begin{theorem}
\label{thm:main} 
Let $K$, $M$, $\tau$ and $M^\tau$ be as above. Let $\rho:\pi_1(M)\to \operatorname{SL}(2,\mathbf C)$ and $\rho^\tau:\pi_1(M^\tau)\to \operatorname{SL}(2,\mathbf C)$
be lifts of the holonomy, with $\operatorname{trace}(\rho(\mu))=\operatorname{trace}(\rho^\tau(\mu))$. Then
$$
\operatorname{tor}(M,\rho)=\operatorname{tor}(M^\tau,\rho^\tau).
$$
 \end{theorem}

This is not true for any representation of $\pi_1(M)$. Wada proved in \cite{Wada} that the twisted Alexander polynomials could be used to distinguish mutant knots. 
N. Dunfield, S. Friedl and N. Jackson \cite{DFJ} computed the torsion for the representation $\rho$ twisted by the abelianization map 
(namely, the corresponding twisted Alexander polynomials) and proved that it distinguishes mutant knots. However, the 
evaluation at $\pm 1$ of these polynomials gives numerical evidence of Theorem~\ref{thm:main}.

In  \cite{MenalPorti}
 we prove that when we consider $\sigma_{2n}:\operatorname{SL}(2,\mathbf C)\to \operatorname{SL}(2n,\mathbf C)$ the $2n$-dimensional irreducible representation, then
the composition $\sigma_{2n}\circ\rho$ is acyclic, thus its torsion is well defined. We have computed that the torsion of $\sigma_{4}\circ\rho$ distinguishes
the Conway  and the Kinoshita-Terasaka mutants, see Section~\ref{sec:concluding}.

The paper is organized as follows. In Section~\ref{sec:mutation}, we discuss the basic constructions for Reidemeister torsion and representations of mutants, and we give a sufficient criterion
in Proposition~\ref{prop:criterion} for invariance of 
the torsion under mutation. The sufficient criterion of Proposition~\ref{prop:criterion} is stated in terms of the action of $\tau$ on the cohomology of $S$ with twisted coefficients. 
 This is applied for the proof when $\operatorname{trace}(\rho(\mu))=-2$ in Section~\ref{sec:trace-2}. The proof when
$\operatorname{trace}(\rho(\mu))=+2$ in Section~\ref{sec:trace+2} is  different, because the involved cohomology groups are different in each situation.
In Section~\ref{section:example} we compute an example, the Kinoshita-Terasaka and Conway mutants, and
Section~\ref{sec:concluding} is 
devoted to further discussion.

\section{Mutation}
\label{sec:mutation}

Let $B_1$ and $B_2$ denote the components of $S^3\setminus {\mathcal N}(C)$, so that the pairs $(B_i,B_i\cap K)$ are tangles with two strings. The 
exterior of the knot is denoted by
$$
M=S^3\setminus \mathcal N(K).
$$
We also denote
$$
S=C\cap M,\quad M_1=M\cap B_1,\quad \textrm{ and }\quad M_2=M\cap B_2.
$$
Write a commutative diagram for the inclusions:
 $$
\xymatrix{ 
C \ar  [r]^{i_1}  \ar[d]_{i_2} & B_1 \ar[d] \\
B_2 \ar[r] & S^3
} 
\qquad\textrm{ and }\qquad
\xymatrix{ 
S \ar  [r]^{i_1}  \ar[d]_{i_2} & M_1 \ar[d] \\
M_2 \ar[r] & M
} 
$$
so that $\pi_1(M$) is an amalgamated product
$$
\pi_1(M)=\pi_1(M_1) *_{\pi_1(S)} \pi_1(M_2).
$$ 
Let $\rho_0$, $\rho_1$ and $\rho_2$ denote the restriction of $\rho$ to $\pi_1(S)$, $\pi_1(M_1)$, and $\pi_1(M_2)$, respectively, so that 
$$
 \rho_1\circ i_{1*}= \rho_2\circ i_{2*}= \rho_0.
$$
Using the notation
$$
\rho_0^a(\gamma)=a\rho_0(\gamma) a^{-1}, \qquad\forall \gamma\in\pi_1(S),
$$
there exists $a\in\operatorname{SL}(2,\mathbf C)$ which is unique up to sign \cite{CL,Ruberman,Tillmann}
such that
$$
\rho_0^a\circ \tau_*= \tau^*\circ \rho_0^a=\rho_0.
$$
Notice that $a\in\operatorname{SL}(2,\mathbf C)$ corresponds to a rotation of order two in hyperbolic space, therefore, 
$a$ is conjugate to
\begin{equation}
\label{eqn:a} 
a \sim
\begin{pmatrix}
   i & 0 \\
   0 & -i
\end{pmatrix}
.
\end{equation}

 To construct the representation of $\pi_1(M^{\tau})$, we also use the amalgamated product structure with the same inclusion $i_1$,
 but with $i_2\circ\tau$ instead of $i_2$. The representation $\rho^\tau:\pi_1(M^\tau)\to \operatorname{SL}(2,\mathbf C)$
is then defined by
$$
\rho^\tau\vert _{\pi_1(M_1)}=\rho_1\qquad \textrm{ and } \qquad \rho^\tau\vert _{\pi_1(M_2)}=\rho_2^a,
$$
because $\rho_2^a\circ(i_2\circ\tau)_*= \rho_0^a\circ\tau_*=\rho_0=\rho_1\circ i_{1*}$.

\subsection{Cohomology with twisted coefficients}

To set notation, we   recall the basic construction of  cohomology with twisted coefficients. Let $X$ be a compact $CW$-complex and $\rho:\pi_1(X)\to \operatorname{SL}(2,\mathbf C)$ a representation.
The cellular chains of its universal covering are denoted by $C_*(\tilde X; \mathbf Z)$, which is a chain complex of left $\mathbf Z[\pi_1(X)]$-modules of finite type. The cochains with twisted 
coefficients are then  
\begin{equation*}
 C^*(X;\rho)=\hom_{\mathbf Z[\pi_1(X)]}( C_*(\tilde X; \mathbf Z),    \mathbf C^2_{\rho} ),
\end{equation*}
where $\mathbf C^2=\mathbf C^2_{\rho}$ is viewed as a  left $\mathbf Z[\pi_1(X)]$-module by the action induced by $\rho$.
The corresponding  cohomology groups are denoted by
$$
 H^*(\vert X\vert ;\rho),
$$
as they only depend on the underlying topological space $\vert X\vert$ of the $CW$-complex $X$. We shall mainly work with aspherical spaces, in this case 
the homology or cohomology of $X$ with twisted coefficients is naturally isomorphic to the group cohomology of $\pi_1(X)$.

We shall also be interested in the de Rham cohomology. Assuming $N$ is a smooth manifold, let $E(\rho)=\tilde N\times \mathbf C^2/\pi_1(N)$ denote the flat bundle with monodromy $\rho$.
The space of $p$-forms valued on $E(\rho)$ is $\Omega^p(N;\rho)=\Gamma(      \Lambda^p T N^*\otimes   E(\rho) )$. The de Rham cohomology of $(\Omega^*(N;\rho),d)$ is naturally isomorphic to $
 H^*(N ;\rho),
$

Many properties of cohomology without coefficients hold true when we have twisted coefficients: Mayer-Vietoris, long exact sequence of the pair, etc.
Poincar\'e duality is discussed in Section~\ref{subsection:perfectpairing}.

\subsection{Mayer-Vietoris exact sequences with twisted coefficients}

We will use Mayer-Vietoris for the pair $(M_1,M_2)$ to compute the torsion of $M$ and of $M^\tau$.
For this, we need to compute some cohomology groups. We start with the planar surface with four boundary components $S=M_1\cap M_2$:

\begin{lemma}
\label{lem:cohomS}
$H^i(S;\rho_0)=0$ for $i\neq 1$ and $H^1(S;\rho_0)\cong \mathbf C^4$.
 \end{lemma}

\begin{proof}
 Firstly $H^0(S,\rho_0)\cong H^0(\pi_1(S),\rho_0)$ is isomorphic to the subspace of $\mathbf C^2$ of elements that are fixed by $\rho_0(\pi_1(S))$, hence it vanishes
because $\rho_0$ is an irreducible representation.
On the other hand,   $H^i(S,\rho_0)=0$ for $i\geq 2$, because $S$ has the homotopy type of a graph. Finally
$$
\dim_ {\mathbf C} H^1(S;\rho_0)=-\chi(S) \dim (\mathbf C^2)=4.
$$
\end{proof}

\begin{lemma}
\label{lem:cohomM_i}
For $k=1,2$, $H^j(M_k;\rho_0)=0$ for $j\neq 1$ and $H^1(M_k;\rho_0)=\mathbf C^2$.
 \end{lemma}

\begin{proof}
By Mayer-Vietoris, and using the fact that $H^*(M;\rho)=0$, then 
 $$
H^j(M_1;\rho_1)\oplus H^j(M_2;\rho_2) \cong H^j(S;\rho_0).
$$
The lemma follows from Lemma~\ref{lem:cohomS}, because $\chi(M_k)=\frac12\chi(S)=1$.
\end{proof}

Mayer-Vietoris for $M$ and $M^\tau$ give the isomorphisms:
\begin{align}
\label{eqn:mv} i_1^*\oplus i_2^*: H^1(M_1;\rho_1) \oplus H^1(M_2;\rho_2) \to H^1(S;\rho_0), & \\
\label{eqn:mvtau} i_1^*\oplus (i_2\circ \tau)^*: H^1(M_1;\rho_1) \oplus H^1(M_2;\rho_2^a) \to H^1(S;\rho_0). & 
\end{align}

In order to relate $H^*(M_2;\rho_2)$ and $H^*(M_2;\rho_2^a)$, we use the composition with $a$,
where $\rho_2^a$ denotes $\rho_2$ conjugated by $a$.
Namely, recall that 
$$
C^*(M_2;\rho_2)=\hom_{\mathbf Z[\pi_1(M_2)]}( C_*(\widetilde M_2;\mathbf Z), \mathbf C^2_{\rho_2}).
$$
Define:
$$
\begin{array}{rcl}
 a_* :C^*(M_2;\rho_2) & \to &  C^*(M_2;\rho_2^a) \\
   \theta & \mapsto & a\circ\theta
\end{array}
$$
It is  straightforward to check that this defines an isomorphism of complexes. Thus 
$a_*: H^1(M_2;\rho_2) \to   H^1(M_2;\rho_2^a)$ is an isomorphism and we have a commutative diagram:
$$
\xymatrix{ 
H^1(M_2;\rho_2) \ar  [r]^{i_2^*}  \ar[d]_{a_*} & H^1(S;\rho_2) \ar[d]^{ a_*\circ\tau^*} \\
H^1(M_2;\rho_2^a) \ar[r]^{(i_2\circ\tau)^*} & H^1(S;\rho_0)
} 
$$
Write $\tau_a^*= a_*\circ\tau^*=\tau^*\circ a_*$. Since $a^2=-\operatorname{Id}$ (see Equation~(\ref{eqn:a})) and since $\tau^2=\operatorname{Id}$,
we have:
\begin{equation}
 \label{eqn:tau2=-id}
  (\tau_a^*)^2=-\operatorname{Id}.
\end{equation}

\subsection{Reidemeister torsions}

Let $X$ be a compact CW-complex equipped with a representation 
$$
\rho:\pi_1(X)\to\operatorname{SL}(2,\mathbf C).
$$
When $H^*(\vert X\vert;\rho)=0$, the Reidemeister torsion can be defined, and it is an invariant of $X$, up to subdivision,  and the conjugacy class of $\rho$.
We will not recall the definition, that can be found in \cite{MilnorBullAMS,TuraevKT}, for instance. There are two main issues for the torsion we are interested in. 
Firstly, the torsion is only defined up to sign, but since we consider a two dimensional vector space, it is sign defined, hence a nonzero complex number.
Equivalently, any choice of homology orientation for Turaev's refined torsion \cite{TuraevKT} gives the same result. Secondly, since we are working with
three and two-dimensional manifolds, the PL-structure is not relevant. Thus, for a two and three-dimensional manifold $X$ and an acyclic representation $\rho:\pi_1(X)\to \operatorname{SL}(2,\mathbf C)
$, its torsion is denoted by 
$$
\tor(\vert X\vert ,\rho)\in\mathbf C\setminus \{0\}.
$$
When $\rho$ is not acyclic, then we can also use the Reidemeister torsion provided we specify a basis for $H^*(\vert X\vert ;\rho)$.

Choose $b_i$ a basis for $H^1(M_i;\rho_i)$ as $\mathbf C$-vector space. In particular $a_*(b_2)$ is a basis for  $H^1(M_2;\rho_2^a)$.
By Milnor's formula \cite{MilnorBullAMS} for the torsion of a long exact sequence applied to (\ref{eqn:mv}) and (\ref{eqn:mvtau}):
\begin{align*}
 \tor(M,\rho)=&\pm \frac{\tor(M_1,\rho_1,b_1) \tor(M_2,\rho_2,b_2)} {\tor(S,\rho_0, i_1^*(b_1)\sqcup i_2^*(b_2)) }  \\
 \tor(M^\tau,\rho^\tau)=&\pm \frac{\tor(M_1,\rho_1,b_1) \tor(M_2,\rho_2^a,a_*(b_2))}{\tor(S,\rho_0, i_1^*(b_1)\sqcup (i_2\tau)^*a^*(b_2)) }.
\end{align*}
Here $\sqcup$ denotes the disjoint union of basis. Notice that Milnor works with torsions up to sign in \cite{MilnorBullAMS}, but its formalism applies
even with sign.
Since $\tau_a^*=\tau^*\circ a_*$, we deduce
\begin{equation}
 \label{eqn:ratiotorsions}
\frac{\tor(M,\rho)}{\tor(M^\tau,\rho^\tau) }=  \det (  i_1^*(b_1)\sqcup \tau_a^* (i_2^*(b_2)) , i_1^*(b_1)\sqcup i_2^*(b_2) ).
\end{equation}
Namely, the determinant of the matrix whose entries are the coefficients of the basis  $i_1^*(b_1)\sqcup  \tau_a^* (i_2^*(b_2))$
with respect to $i_1^*(b_1)\sqcup i_2^*(b_2)$.

The following is a sufficient criterion for invariance of torsion with respect to mutation.

\begin{proposition}
 \label{prop:criterion}
If $\tau_a^* : H^1(S;\rho_0)\to H^1(S;\rho_0)$ leaves invariant the  image of $i_2^*: H^1(M_2;\rho_2)\to H^1(S,\rho_0)$, then
$\tor(M,\rho)=\pm \tor(M^\tau,\rho^\tau)$.
\end{proposition}

\begin{proof}
Since $(\tau_a^*)^2= -\Id$
 by Formula (\ref{eqn:tau2=-id}), $\tau_a^*$ diagonalizes with eigenvalues $\pm i$. Hence, assuming that $\tau_a^*$
leaves invariant the  image of $i_2^*$,
the matrix in 
Equation~(\ref{eqn:ratiotorsions}) is conjugate to
$$
\begin{pmatrix}
    1 & 0 & 0 & 0 \\
    0 & 1 & 0 & 0 \\
    0 & 0 & \pm i & 0 \\
    0 & 0 & 0 & \pm i
\end{pmatrix},
$$
hence it has determinant $\pm 1$.
\end{proof}

\section{Invariance when $\operatorname{trace}(\mu)=-2$.}
\label{sec:trace-2}

We discuss first the case where $\operatorname{trace}(\mu)=-2$.  
The proof has three parts. In Subsection~\ref{subsection:perfectpairing}
we consider a perfect pairing on $
H^1(S;\rho_0)$ (pulling back the cup product on $H^1(\partial M_k;\rho_k)$
by the  isomorphism $H^1(S;\rho_0)\cong H^1(\partial M_k;\rho_k)$).
We show that, for $k=1$, $2$,  the images of $i_k^*:H^1(M_k;\rho_k)\to H^1(S;\rho_0)$ are isotropic subspaces.
Then in Subsection~\ref{subsec:quadric} we analyze properties of isotropic planes  of 
$H^1(S;\rho_0)\cong\mathbf C^4$, that are viewed as lines in a ruled quadric in $\mathbf P^3$.
The properties of this ruled quadric are used in a deformation argument in Subsection~\ref{subsection:firstdfm argument}
to conclude the proof when $\operatorname{trace}(\mu)=-2$.

\subsection{A perfect pairing}
\label{subsection:perfectpairing}

For a closed oriented $n$-manifold $N^n$ and a representation $\rho:\pi_1(N^n)\to \operatorname{SL}(2,\mathbf C)$, there is a nondegenrate pairing:
\begin{equation}
 \cup: H^k(N^n;\rho)\times H^{n-k}(N^n ; \rho)\to H^n(N^n ; \mathbf C^2\otimes \mathbf C^2)\to  H^n(N^n ; \mathbf C)\cong \mathbf C,
\end{equation}
which is the composition of the cup product and the map induced from the $\rho$-invariant pairing
$$
\begin{array}{rcl}
    \mathbf C^2 \otimes  \mathbf C^2  & \to & \mathbf C \\
       \begin{pmatrix}
         a \\ b
       \end{pmatrix}
	\otimes
\begin{pmatrix}
         c \\ d
       \end{pmatrix}
& 
\mapsto
&
\det
      \begin{pmatrix}
       a & c \\
       b & d
      \end{pmatrix}
\end{array}.
$$
This pairing  is bilinear, nondegenerate (Poincar\'e duality) and natural. In addition, when $n=2$ and $k=1$ it is symmetric (because both the usual cup product and the determinant are antisymmetric).
In particular, we have two symmetric pairings, for $k=1,2$:
\begin{equation}
 \label{eqn:pairingmk}
\cup_k: H^1(\partial M_k;\rho_k)\times H^1(\partial M_k; \rho_k)\to H^2(\partial M_k;\mathbf C)\cong \mathbf C.
\end{equation}

The following lemma assumes that $\operatorname{trace}(\rho(\mu))=-2$, as the whole section, though it only requires $\operatorname{trace}(\rho(\mu))\neq 2$.

\begin{lemma}
\label{lemma:Spartial} $ H^*(\partial S;\rho_0)=0$.
\end{lemma}

\begin{proof}
 Since $\operatorname{trace}(\rho(\mu))=-2$, $\rho_0(\mu)$ has no nontrivial fixed vectors in $\mathbf C^2$. Thus $H^0(\partial S;\rho_0)=0$ and, by duality,
 $ H^*(\partial S;\rho_0)=0$.
\end{proof}

\begin{lemma}
\label{lemma:SpartialMk}
For $k=1,2$, the inclusion map induces an isomorphism 
\begin{equation}
\label{eqn:isoS} H^1(S;\rho_0)  \cong  H^1(\partial M_k; \rho_k).
\end{equation}
\end{lemma}

\begin{proof}
This  follows from the  Mayer-Vietoris sequence applied to $S$ and $\overline{\partial M_k\setminus S}$,
which  is the union of the two annuli 
 around the arcs of $K\cap B_k$ that have the homotopy type of a component of $\partial S$. Hence by Lemma~\ref{lemma:Spartial},  
$H^*(\partial M_k\setminus S;\rho_k) =0$.
\end{proof}

Pulling back the pairings (\ref{eqn:pairingmk}) by the 
isomorphism (\ref{eqn:isoS}), we obtain two symmetric perfect pairings
\begin{equation}
 \label{eqn:SS}
\cup_k':H^1(S;\rho_0)\times H^1(S;\rho_0)\to  \mathbf C.
\end{equation}

\begin{lemma}
\label{lemma:cup1cup2} 
Both pairings are the same: $\cup_1'=\cup_2'$.
\end{lemma}

\begin{proof}
We use de Rham cohomology for the proof.
We first claim that every cohomolgy class in $H^1(S;\rho_0)$ is represented by a form compactly supported in the interior of $S$, using that  $H^1(\partial S;\rho_0)=0$.
Namely, let $\omega\in \Omega^1(S;\rho_0)$ be any closed form and $U\subset S$ a tubular neighborhood of $\partial S$. Since $H^1(\partial S;\rho_0)=0$,
the restriction of $\omega$ to $ U$ is exact: there exists a section $s\in \Omega^0(U;\rho_0)$ satisfying
$$
\omega\vert_U=d s,
$$
where $d$ is the differential. Consider a smooth function $\varphi: U\to [0,1]$ that extends smoothly to zero on $S\setminus U$ and equals one in a smaller neighborhood of
$\partial S$. Then the form
$$
\omega- d(\varphi s)
$$ 
is cohomologous to $\omega$ and is supported on a compact subset of the interior of $S$.

Given two closed forms $\omega_1,\omega_2\in \Omega^1(S;\rho_0)$ with compact support in the interior of $S$, we may view them as smooth forms 
 $ \tilde \omega_1,\tilde \omega_2$
on $\partial M_k$ by extending them trivially. As the isomorphism $  H^1(\partial M_k; \rho_k)\cong H^1(S;\rho_0) $ is induced by restriction, it maps 
the cohomology classes $[\tilde \omega_1], [\tilde \omega_2]\in  H^1(\partial M_k; \rho_k)$   to $[\omega_1], [\omega_2]\in H^1(S;\rho_0)$.
 Thus
the cup product $\cup_k'$ of their cohomology classes is
$$
[\omega_1]\cup_k' [\omega_2]= [\tilde \omega_1]\cup_k [\tilde \omega_2]= \int_{\partial M_k} \det (\tilde\omega_1\wedge\tilde \omega_2),
$$
where $\det (\tilde\omega_1\wedge \tilde\omega_2)$ denotes the pairing associated to the determinant evaluated at the exterior product of the forms.
By construction, the support of $\tilde\omega_1$ and $\tilde\omega_2$ is contained in $S$, thus the previous integral can be evaluated on $S$ instead of $\partial M_k$,
and in particular it does not depend on $k$.
\end{proof}

By the previous lemma, we may omit the subindex and just write 
$$
\cup= \cup_1'=\cup_2'.
$$

\begin{lemma}
\label{lem:isotropic}
The image of $i_2^*: H^1(M_2;\rho_2)\to H^1(S;\rho_0)$ is an isotropic plane for the product (\ref{eqn:SS}). 
\end{lemma}

\begin{proof}
The fact that the image is a plane follows from Lemma~\ref{lem:cohomM_i} and its proof.
Then, by construction the lemma is equivalent to saying that 
the image of $H^1(M_k;\rho_2)\to H^1(\partial M_k;\rho_k)$ is an isotropic subspace. This is well known \cite{Hodgson, Sikora}, but we sketch the argument for completeness.

Using the long exact sequence of the pair $(M_k,\partial M_k)$, we have a  
 commutative diagram with exact rows:
 \begin{equation}
\label{eqn:diaramPD}
\xymatrix{ 
H^1(M_k;\rho_k) \ar  [r]^{j_k^*} \ar@{}[d]^*++{\!\!\!\!\times}  & H^1(\partial M_k;\rho_k) \ar[r]^{\Delta}  \ar@{}[d]^*++{\!\!\!\!\times} & H^2(M_k,\partial M_k;\rho_k)  \ar@{}[d]^*++{\!\!\!\!\times}  \\
 H^2(M_k,\partial M_k;\rho_k) \ar[d] &   H^1(\partial M_k;\rho_k) \ar[l]_{\Delta} \ar[d] & H^1(M_k;\rho_k) \ar  [l]_{j_k^*}  \ar[d]\\
\mathbf C & \mathbf C & \mathbf C 
} 
\end{equation}
where the columns denote the pairings. 
This implies 
$$
j_k^*(a)\cup b= a\cup \Delta (b),\qquad \forall a\in H^1(M_k;\rho_k)\textrm{ and } b\in H^1(\partial M_k;\rho_k).
$$
Hence
$$
j_k^*(a)\cup j_k^*(b)= a\cup \Delta (j_k^*(b))=0,\qquad \forall a,b\in H^1(M_k;\rho_k),
$$
because $\Delta\circ j_k^*=0$, and we are done.
\end{proof}

\subsection{Finding isotropic planes with the ruled quadric}
\label{subsec:quadric}

Let $\mathbf P ^3$ denote the projective space on $H^1(S;\rho_0)\cong \mathbf C^4$. Isotropic planes of  $H^1(S;\rho_0)$ are in bijection with projective lines in the quadric
$$
Q=\{ x\in\mathbf P ^3\mid x\cup x=0\}.
$$
Since $\cup$ is a nondegenrate paring, $Q$ is the standard quadric, which is a ruled surface, with two rulings. We recall next its basic properties.

\begin{proposition}
\label{prop:ruling} 
There are two disjoint families of projective lines $\mathcal L_+$ and $\mathcal L_-$ in $Q$ such that:
\begin{itemize} 
\item[(i)]	 Every line in $Q$ belongs to either $\mathcal L_+$ or $\mathcal L_-$.
\item[(ii)]	 Every point in $Q$ belongs to precisely one line in $\mathcal L_+$ and one in  $\mathcal L_-$. 
\item[(iii)]	 Two lines in $Q$ intersect if, and only if, one is in $\mathcal L_+$ and the other one is in $\mathcal L_-$.
\item[(iv)]	 Embedding those spaces of lines in the projective Grassmannian, $\mathcal L_+\cong \mathcal L_-\cong\mathbf P^1$.
\end{itemize}
\end{proposition}

We shall also use the action of $\operatorname{SO}(4,\mathbf C)$, the isometry group of $H^1(S;\rho_0)$ equipped with the quadratic form $\cup$
in (\ref{eqn:SS}). Recall the isomorphism:
$$
\operatorname{SL}(2,\mathbf C)\times \operatorname{SL}(2,\mathbf C)/\pm (\Id,\Id)\cong \operatorname{SO}(4,\mathbf C).
$$
After projectivizing, this induces an isomorphism:
$$
\operatorname{PSL}(2,\mathbf C)\times \operatorname{PSL}(2,\mathbf C) \cong \operatorname{PSO}(4,\mathbf C).
$$

\begin{proposition}
\label{prop:product}
The action of $\operatorname{PSL}(2,\mathbf C)\times \operatorname{PSL}(2,\mathbf C)$ is equivalent to the 
product action on $\mathcal L_+\times \mathcal L_-\cong \mathbf P^1\times\mathbf P^1$ (by an equivalence that preserves the product).  
\end{proposition}

See \cite{FultonHarris} for a proof.
It follows from this proposition that $\operatorname{PSL}(2,\mathbf C)\times \{\Id\}$ acts trivially on $\mathcal L_-$, and 
 $ \{\Id\}\times \operatorname{PSL}(2,\mathbf C)$ acts trivially on $\mathcal L_+$.

Consider the involutions $\tau_1$, $\tau_2$ and $\tau_3$ of the Conway sphere $C$, as in Figure~\ref{fig:involutions}.
For $i=1,2,3$  let $a_i\in\operatorname{SL}(2,\mathbf C)$, satisfy $\rho_0\circ\tau_{i*}= \rho_0^{a_i}$,
 and define
$\tau_{a_i}^*=a_{i*}\circ \tau^*_i$ the corresponding actions on 
$H^1(S;\rho_0)$.

\begin{lemma}
The induced maps  $\tau_{a_i}^*$ on $\mathbf P^3$ lie in one factor
$\operatorname{PSL}(2,\mathbf C)\times \{\Id\}$ or  $ \{\Id\}\times \operatorname{PSL}(2,\mathbf C)$.
In addition all 
$\tau_{a_1}^*$, $\tau_{a_2}^*$ and $\tau_{a_3}^*$
lie in the same factor.
\end{lemma}

\begin{proof}
By looking at the action on  $H^1(S;\rho_0)\cong \mathbf C^4$, we have $(\tau_{a_i}^*)^2=(a_{i*})^2\circ (\tau^*_i)^2=-\Id$, by (\ref{eqn:tau2=-id}).
This implies that $\tau_{a_i}^*$ projects to an involution of $\mathbf P^3$ preserving $\cup$, 
hence to an involution in
$\operatorname{PSL}(2,\mathbf C)\times \operatorname{PSL}(2,\mathbf C)$. 
Notice that if an involution of $\operatorname{PSL}(2,\mathbf C)\times \operatorname{PSL}(2,\mathbf C)$ 
is nontrivial on each factor, it lifts to an element of  $\operatorname{SL}(2,\mathbf C)\times \operatorname{SL}(2,\mathbf C)$
whose square is $-(\Id,\Id)$, hence to an involution of $\operatorname{SO}(4,\mathbf C)\cong \operatorname{SL}(2,\mathbf C)\times \operatorname{SL}(2,\mathbf C)/\pm (\Id,\Id)$.
 As $(\tau_{a_i}^*)^2=-\operatorname{ Id}\in \operatorname{SO}(4,\mathbf C)$, 
we deduce that each $\tau_{a_i}^*$ projects to an involution of one of the factors
of $\operatorname{PSL}(2,\mathbf C)\times \operatorname{PSL}(2,\mathbf C)$,
and is trivial on the other factor. This proves the first assertion of the lemma.
For the last assertion, just use that $\tau_{a_1}^* \tau_{a_2}^*= \pm \tau_{a_3}^*$.
\end{proof}

Hence, up to permuting $\mathcal L_-$ and $\mathcal L_+$, we get:

\begin{corollary}
 \label{coro:tauasigmab}  
\begin{itemize}
 \item[(i)] $\tau_{a_1}^*$, $\tau_{a_2}^*$ and $\tau_{a_3}^*$ act trivially on $\mathcal L_-$.
\item[(ii)]  There is no point in $\mathcal L_+$ fixed by all  $\tau_{a_1}^*$, $\tau_{a_2}^*$ and $\tau_{a_3}^*$.
\end{itemize}
 \end{corollary}

Notice that, for assertion (ii), we use  that the subgroup of $\operatorname{PSL}(2,\mathbf C)$ consisting of three involutions and the identity
(also called the 4-Klein group) has no global fixed point in $\mathbf P ^1$.

Let $\operatorname{Im}$ denote the image. Since $\operatorname{Im}(i_1^*)\oplus \operatorname{Im}(i_2^*)= H^1(S;\rho_0)$, from    Proposition~\ref{prop:ruling} (iii) we deduce:

\begin{corollary}
\label{cor:imi1i22}
Either both $\operatorname{Im}(i_1^*)$ and $ \operatorname{Im}(i_2^*)$  belong to $\mathcal L_-$, or they both belong to $\mathcal L_+$.
\end{corollary}

Either $\operatorname{Im}(i_2^*)$ belongs to $\mathcal L_-$ and, by Corollary~\ref{coro:tauasigmab}, we may apply Proposition~\ref{prop:criterion},
 or  $\operatorname{Im}(i_2^*)$ belongs to $\mathcal L_+$.
To get rid of this last case we will use a deformation argument. The idea is to deform the hyperbolic  structure on $M_2$, so that it matches with another tangle
which is invariant under the involutions $\tau_i$. By Corollary~\ref{coro:tauasigmab}(ii), the tangle invariant by the involutions satisfies $\operatorname{Im}(i_1^*)\in \mathcal L_-$,
hence $\operatorname{Im}(i_2^*)\in \mathcal L_-$ for the deformed structure on $M_2$,
by Corollary~\ref{cor:imi1i22}. Then we shall use a continuity argument
to have the same conclusion for the initial structure on $M_2$.
 Next subsection is devoted to this deformation argument.

\subsection{A deformation argument}
\label{subsection:firstdfm argument}

Let $A= \overline{\partial M_2\setminus S}$ be the pair of annuli, one around each arc of $K\cap B_2$. The pair $(M_2,A)$ is a \emph{pared} manifold.

\begin{definition}
\label{def:pared}
A \emph{pared} manifold is a pair $(N,P)$, where $N$ is a compact oriented 3-manifold, $P\subset \partial N$ is a union of tori and annuli, such that 
\begin{itemize}
\item[(i)]  no two components of $P$ are isotopic in $\partial N$,
\item[(ii)] every abelian noncyclic subgroup of $\pi_1(N)$ is conjugate to a subgroup of a component or $P$, and
\item[(iii)] there are no essential annuli $(S^1\times [0,1],S^1\times \partial [0,1])\to (M,P)$.
\end{itemize}
\end{definition}

We say that a pared manifold $(N,P)$ is hyperbolic when the interior of $N$ admits a complete hyperbolic structure with cusps at $P$. The rank of the cusp is one for an annulus, and two for a torus.

\begin{lemma}
\label{lemma:tangle_hyp}
 There exists a pared manifold $(M_3, A')$, such that
\begin{enumerate}
 \item $(M_3,A')$ is obtained from a $2$-tangle: namely $M_3$ is the exterior of two properly embedded 
arcs in a 3-ball,   $A'$ are the annuli around the arcs of the tangle, and $A'\cup S=\partial M_3$.
 \item For $i=1,2,3$, $\tau_i:S\to S$ extends to an involution of $(M_3, A')$.
 \item The pared manifolds $(M_3,A')$ and $(M_2\cup M_3, A\cup A')$ are both  hyperbolic.
\end{enumerate}
\end{lemma}

\begin{proof}[Proof of Lemma~\ref{lemma:tangle_hyp}]
We take $M_3$ to be the exterior of a simple 2-tangle that is symmetric with respect to $\tau_1$, $\tau_2$ and $\tau_3$. Here simple means that $M_3$ is irreducible, $\partial$-irreducible, atoroidal and anannular.
In \cite{Wu}, Wu gives a criterion for deciding when a rational tangle is simple and provides  examples of simple tangles with the required symmetries.
In particular, the pared manifold  $(M_3,A')$ admits a hyperbolic structure with totally geodesic boundary in $S=\partial M_3\setminus A'$ (and rank one cusps in $A'$).
Since $(M_3,A')$ is simple and $(M_2,A)$ hyperbolic, standard arguments in 3-dimensional topology prove that   $(M',T')=(M_2\cup M_3, A\cup A')$ is irreducible, acylindrical, atoroidal and not Seifert fibered 
($S$ should be horizontal in a Seifert fibration), hence hyperbolic.
 \end{proof}

The variety or representations of $\pi_1(M_2)$ to $\operatorname{SL}(2,\mathbf C)$ is denoted by
$$
R(M_2)=\hom ( \pi_1(M_2), \operatorname{SL}(2,\mathbf C)),
$$
and it is an algebraic subset of affine space $\mathbf C^N$.

\begin{lemma}
 \label{lemma:almostequal-2} If $\operatorname{trace}(\rho(\mu))=-2$, then
$\operatorname{Im}(i_2^*)$ belongs to $\mathcal L_-$.
\end{lemma}

\begin{proof}
We connect $\rho_2\in R(M_2)$ to $\rho_2'\in R(M_2)$, a lift of the holonomy representation of $M_2$ that matches with the tangle $M_3$ of Lemma~\ref{lemma:tangle_hyp}, which is a symmetric tangle.
Namely we want to  find a path or representations 
$$
\begin{array}{rcl}
  [0,1]&\to &R(M_2) \\
      t & \mapsto & \varphi_t
\end{array}
$$ 
that satisfies:
\begin{itemize}
 \item[(i)] $\varphi_0=\rho_2$.
  \item[(ii)] $\forall t\in [0,1]$, $\varphi_t$ is the lift of the holonomy of a hyperbolic structure on $M_2$, with rank one cusps at the arcs $K\cap B_2$, and satisfying 
	  \[\operatorname{trace}(\varphi_t(\mu))=-2.\]
  \item[(iii)] $\forall t\in [0,1]$, $\dim H_1(M_2;\varphi_t)=2$.
  \item[(iv)] $\varphi_1=\rho_2'$ is the lift of the holonomy of a hyperbolic structure on $M_2$ that matches with $M_3$  in Lemma~\ref{lemma:tangle_hyp}.
\end{itemize}
 
Assuming we have this path of representations, then 
since $M_3$ is $\tau_1$ and $\tau_2$-invariant, the image of $i_3^*:H^1(M_3;\varphi_1)\to H^1(S;\varphi_1\vert_{\pi_1(S)})$ is a subspace $\tau_{a_i}^*$-invariant.
Hence the image of $i_3^*$ must be contained in $\mathcal L_-$, by Corollary~\ref{coro:tauasigmab} (ii). This implies that for this hyperbolic structure 
$$
\operatorname{Im}\left(i_2^*: H^1(M_2;\rho_2')\to H^1(S;\rho_2'\vert_{S})\right) \in \mathcal L_-,
$$
 by Corollary~\ref{cor:imi1i22}. Now, since there exists the path $\varphi_t$,
the ruled quadric of $H^1(S_0;\varphi_t)$ is also deformed continuously (notice that as $\varphi_t\vert_{S_0}$ is irreducible and $\varphi_t$ of a meridian 
  has trace $-2$, by (iii),  Lemmas~\ref{lem:cohomS}, \ref{lemma:Spartial}, and \ref{lem:isotropic} apply to $H^1(S_0;\varphi_t)$).
Hence along the deformation, the image of
$i_2^*$ is contained in $\mathcal L_-$, as $\mathcal L_+\cap \mathcal L_-=\emptyset$. Hence 
$$
\operatorname{Im}\left(i_2^*: H^1(M_2;\rho_2)\to H^1(S;\rho_0) \right)\in \mathcal L_-,
$$
as claimed.

Let us justify the existence of the path $\phi_t$ between $\rho_2$ and $\rho_2'$.
If both $\rho_2(\pi_1(M_2))$  and  $\rho_2'(\pi_1(M_2))$ are geometrically finite, then they can be connected along the space of geometrically finite structures of the pared manifold, because by Ahlfors-Bers theorem 
this space is isomorphic to the Teichm\"uller space of $S$, cf.\ \cite{Otal}.  In addition, this is an open subset of the variety of representations, and since the dimension of 
de cohomology is upper semi-continuous (it can only jump in a Zariski closed subset), (iii) can be achieved by avoiding a proper  Zariski closed  subset (hence of real codimension $\geq 2$). 
If any of $\rho_2(\pi_1(M_2))$  and  $\rho_2'(\pi_1(M_2))$  is not geometrically finite, then it lies in the closure of geometrically finite structures (cf.\ \cite{OtalFib} though this is a 
particular case of the density theorem), thus there is still a path in the space
of representations satisfying (ii) and (iii). 
 \end{proof}

By Lemma~\ref{lemma:almostequal-2}, Corollary~\ref{cor:imi1i22}  and Proposition~\ref{prop:criterion}, 
$$
\tor(M,\rho)=\pm \tor(M^\tau,\rho^\tau).
$$
We shall prove that there is also equality of signs:

\begin{proposition}
 \label{proposition:equal-2} If $\operatorname{trace}(\rho(\mu))=-2$, then
$$
\tor(M,\rho)= \tor(M^\tau,\rho^\tau).
$$
\end{proposition}

\begin{proof}
To remove the sign ambiguity, we use again the deformation $\varphi_t$ of the proof of Lemma~\ref{lemma:almostequal-2}.
Since 
$\forall t\in [0,1]$,  $\varphi_t$ satisfies the sufficiency criterion of Proposition~\ref{prop:criterion},
 the eigenvalues of $\tau_{a_i}^*$ restricted to the image of $i_2^*$ belong to $\{\pm i\}$, and they 
do not change as we deform $t$, hence 
the determinant of   $\tau_{a_i}^*$  restricted to the image of $i_2^*$ is $+1$, because this holds for $\rho_2'=\varphi_1$
(as $M_3$ is $\tau_i$-invariant).
\end{proof}

\section{Invariance when $\operatorname{trace}(\rho(\mu))=+2$.}
\label{sec:trace+2}

When $\operatorname{trace}(\rho(\mu))=2$, then Lemma~\ref{lemma:Spartial} does not apply, hence we cannot use the argument of Section~\ref{sec:trace-2}. 
Recall that 
$$
R(M)=\hom(\pi_1(M),\operatorname{SL}(2,\mathbf C))
$$
denotes the variety of representations of $\pi_1(M)$ in $\operatorname{SL}(2,\mathbf C)$.
We will consider representations $\rho_n\in R(M)$ to which the arguments of Section~\ref{sec:trace-2} apply
and such that $\rho_n$ converges to $\rho$ in $R(M)$, as $n\to\infty$.

Let $\rho\in R(M)$ be a lift of the holonomy with  $\operatorname{trace}(\rho(\mu))=2$. By Thurston's hyperbolic Dehn filling and for $n\in\mathbf N$ large enough, the orbifold 
with underlying space $S^3$, branching locus $K$ and ramification index $n$ is hyperbolic. It induces a representation of $\pi_1(M)$ in $ \operatorname{PSL}(2,\mathbf C)$ 
that lifts to $\rho_n\in R(M)$. The lift satisfies $\operatorname{trace}(\rho_n(\mu))=\pm 2\cos(\pi/n)$, and   there is precisely one lift for every choice of sign. We chose the lift satisfying 
$\operatorname{trace}(\rho_n(\mu))=+ 2\cos(\pi/n)$.

\begin{proposition}[\cite{ThurstonNotes}]
 \label{prop:orbifold}
For $n\in\mathbf N$ large enough, there exist $\rho_n\in R(M)$ which is a lift of the holonomy of the orbifold 
with underlying space $S^3$, branching locus $K$ and ramification index $n$, so that $\rho_n\to \rho$.
\end{proposition}

These orbifolds can also be considered for the mutant knot, and there exist the corresponding mutant representations
$$
\rho_n^\tau\in R(M^{\tau}).
$$
Namely, the lifts of the holonomies of the orbifold structures on $K^\tau$ are the ``mutant representations'' of $\rho_n$.
 Moreover, $\rho_n^\tau\to \rho^{\tau}$.

\begin{lemma}
\label{lem:orbif}
For $n\in\mathbf N$ large enough, $H^*(M;\rho_n)\cong H^*(M^\tau;\rho_n^\tau)=0$ and 
$$
\tor(M,\rho_n)= \tor(M^\tau,\rho_n^\tau).
 $$
\end{lemma}

\begin{proof}

We use semi-continuity of cohomology  to say that $H^*(M;\rho_n)\cong H^*(M^\tau;\rho_n^\tau)=0$.
More precisely,  the dimension of cohomology is an upper semi-continuous function on $R(M)$,
cf.~\cite{Hartshorne}.
Hence, as $\rho$ and $\rho^\tau$ are acyclic, then all representations in a Zariski open subset containing $\rho$ and $\rho^\tau$ 
 are acyclic, and so are $\rho_n$
and $\rho_n^\tau$, as claimed. 

Since $\rho_n$
and $\rho_n^\tau$ are acyclic, 
  $\dim H_1(M_2;\rho_n)=\dim H_1(M_2;\rho_n^\tau)=2 $. 
In addition, up to conjugation,
$$
\rho_n(\mu) \sim \begin{pmatrix} e^{\pi i/n} & 0 \\ 0 & e^{-\pi i/n} \end{pmatrix}.
$$
Hence $\mathbf C^2$ has no 
$ \rho_n(\mu)$-invariant proper subspaces and Lemma~\ref{lemma:Spartial} applies.
 Thus the pairing~(\ref{eqn:SS}), Lemma~\ref{lem:isotropic} and all the results of Section~\ref{subsection:perfectpairing} hold true for $\rho_n$.
To conclude, for  the deformation argument, we use that $\rho_n$ is the lift of the holonomy of an orbifold. Instead of working with pared structures on
$(M_2,A)$, we work with orbifold structures with underlying space the ball $B_2$, branching locus $K\cap B_2$ and branching index $n$. The results on
the space of hyperbolic structures (geometrically finite or infinite) apply, and we may use the
 deformation argument of Lemma~\ref{lemma:almostequal-2}.
\end{proof}

By Proposition~\ref{prop:orbifold} and Lemma~\ref{lem:orbif}, by taking the limit when $n\to\infty$ we get:

\begin{corollary}
\label{cor:equal-2} If $\operatorname{trace}(\rho(\mu))=+2$, then
$$
\tor(M,\rho)= \tor(M^\tau,\rho^\tau).
$$
\end{corollary}

The proofs for $\operatorname{trace}(\rho(\mu))=2$ and $\operatorname{trace}(\rho(\mu))=-2$ are quite different, because $H^1(\partial S_0;\rho_0)$ is non zero when the trace is +2, 
and vanishes when it is -2.
The generic case is 
$\operatorname{trace}(\mu)\neq 2$. The proof of Section~\ref{sec:trace-2} applies to the following situation.

\begin{proposition}
\label{prop:generic}
Let $\rho:\pi_1(M)\to \operatorname{SL}(2,\mathbf C)$ be a representation satisfying:
\begin{enumerate}
  \item $H^1(M;\rho)=0$;
  \item $\operatorname{trace}(\rho(\mu))\neq 2$;
 \item $\rho$ restricted to $\pi_1(S)$ is irreducible;
  \item the representation $\rho$ is in the same irreducible component of $R(M)$ as some representation such that the image of $i_2^*$ is contained in $Q_-$.
\end{enumerate} 
Then $\tor(M,\rho)=\tor(M^\tau,\rho)$.
\end{proposition}

\begin{corollary}
\label{cor:generic}
For a generic representation $\rho$ of the irreducible component of $R(M)$ that contains a lift of the holonomy,
 $\tor(M,\rho)=\tor(M^\tau,\rho^\tau)$.
\end{corollary}

\begin{question}
The holonomy representation of a hyperbolic knot has two lifts to $\operatorname{SL}(2,\mathbf C)$, each one with a different sign for the image of the meridian. 
Do they belong to the same irreducible component of the variety of representations?
\end{question}

This happens to be true for instance if the component of the variety of representations contains a dihedral representation, as this is a ramification point 
or the map from the variety of representations in $\operatorname{SL}(2,\mathbf C)$ to those in $\operatorname{PSL}(2,\mathbf C)$.

\section{Example: Kinoshita-Terasaka and Conway mutants}
\label{section:example}

Let $KT$ and $C$ be the Kinoshita-Terasaka knot and the Conway knot respectively.
It is well known that they are mutant hyperbolic knots. Using the Snap program \cite{CGHN}, based of J. Weeks' SnapPea \cite{Weeks}, we 
have obtained all the necessary information to compute their torsion.

The fundamental groups of these knots have the following presentations:
\begin{eqnarray*}
	 \pi_1(S^3\setminus C)  & = & \langle a b c  \mid abACbcbacBCABaBc, aBcBCABacbCbAbacbc  \rangle, \\
	\pi_1(S^3\setminus KT) & = & \langle a b c  \mid aBCbABBCbaBcbbcABcbbaB, abcACaB \rangle.
\end{eqnarray*}
As usual, capital letters denote inverse.

The image of the holonomy representation is contained in 
$\operatorname{PSL}(2, \mathbf{Q}(\omega))$
where $\mathbf{Q}(\omega)$ is the number field generated by a root $\omega$ of the following polynomial:
\[
p(x) = x^{11} - x^{10} + 3 x^9 - 4 x^8 + 5  x^7 - 8 x^6 + 8 x^5 - 5 x^4 + 6 x^3 - 5 x^2 + 2x -1.
\]
The torsions then are elements of $\mathbf{Q}(\omega)$. 
In order to express elements in $\mathbf{Q}(\omega)$, we use the $\mathbf{Q}$-basis
$(\omega^{10}, \omega^9, \cdots, \omega, 1)$. 
Tables~\ref{table:tra2} and \ref{table:tra-2} give the coefficients of the torsions of $KT$ and $C$ with respect to this $\mathbf Q$-basis. On each table,  the first column
gives the element of the basis. We let  $n$ denote the dimension of the irreducible representation of $\operatorname{SL}(2,\mathbf C)$ used to compute the torsion,
and the tables show the values for $n=2$ (i.e. the standard representation), but also  $n=4$ and $n=6$.
In order to compare them, the coefficients of the torsion for Kinoshita-Terasaka ($KT$) and Conway ($C$) knots are tabulated side by side.
We give a table for each lift of the holonomy, one when the trace of the meridian is $2$ (Table~\ref{table:tra2}) and another when it is $-2$ (Table~\ref{table:tra-2}).

\begin{table}
\begin{small}
 \hspace{-1.6cm}
\begin{tabular}{cc|r|r|r|r|r|r|}
	 \cline{3-8}
	& & \multicolumn{2}{|c|}{ $n=2$} & \multicolumn{2}{|c|}{ $n=4$} & \multicolumn{2}{|c|}{$n=6$}  \\ \cline{1-8}
\multicolumn{2}{ |c| }{  } & \multicolumn{1}{|c|}{$KT$} & \multicolumn{1}{|c|}{$C$} &  \multicolumn{1}{|c|}{$KT$} & \multicolumn{1}{|c|}{$C$} &   \multicolumn{1}{|c|}{$KT$} & \multicolumn{1}{|c|}{$C$}   \\ \cline{1-8}
\multicolumn{2}{ |c| }{ $\omega^{10}$} &$356$ & $356$ & $11112880$ & $11112880$ & $676803770859632$ & $662357458754672$ \\ \cline{1-8}
\multicolumn{2}{ |c| }{ $\omega^{9}$} &$-620$ & $-620$ & $-38963592$ & $-38963592$ & $-640579476284656$ & $-579216259622896$ \\ \cline{1-8}
\multicolumn{2}{ |c| }{ $\omega^{8}$} &$636$ & $636$ & $36107416$ & $36107416$ & $212555254795952$ & $153724448856752$ \\ \cline{1-8}
\multicolumn{2}{ |c| }{ $\omega^{7}$} &$-864$ & $-864$ & $-31579196$ & $-31579196$ & $-990061444305088$ & $-943617945204928$ \\ \cline{1-8}
\multicolumn{2}{ |c| }{ $\omega^{6}$} &$1228$ & $1228$ & $60889040$ & $60889040$ & $1004678681648016$ & $908722528184976$ \\ \cline{1-8}
\multicolumn{2}{ |c| }{ $\omega^{5}$} &$-1080$ & $-1080$ & $-58195768$ & $-58195768$ & $-444238765345264$ & $-349679698188784$ \\ \cline{1-8}
\multicolumn{2}{ |c| }{ $\omega^{4}$} &$780$ & $780$ & $36555000$ & $36555000$ & $482101712163904$ & $424247992815424$ \\ \cline{1-8}
\multicolumn{2}{ |c| }{ $\omega^{3}$} &$-628$ & $-628$ & $-31740272$ & $-31740272$ & $-371824600930944$ & $-320894530449024$ \\ \cline{1-8}
\multicolumn{2}{ |c| }{ $\omega^{2}$} &$428$ & $428$ & $21313180$ & $21313180$ & $51168266257072$ & $15655188602032$ \\ \cline{1-8}
\multicolumn{2}{ |c| }{ $\omega^{1}$} &$-188$ & $-188$ & $-8829332$ & $-8829332$ & $-165869512283168$ & $-152117462516768$ \\ \cline{1-8}
\multicolumn{2}{ |c| }{ $1$ } &$124$ & $124$ & $7476160$ & $7476160$ & $-37602419304496$ & $-50452054740016$ \\ \cline{1-8}
\end{tabular} \end{small}
\caption{\normalsize{Torsions for the lift of the holonomy  with trace of the meridian $2$. The table gives the coefficients of the torsion
of $n$-dimensional representation
(with respect to a $\mathbf Q$-basis for $\mathbf{Q}(\omega)$).}}
\label{table:tra2}
\end{table}

\begin{table}
\begin{footnotesize}
\hspace{-1.8cm}
\begin{tabular}{cc|r|r|r|r|r|r|}
      \cline{3-8}
	& & \multicolumn{2}{|c|}{ $n=2$} & \multicolumn{2}{|c|}{ $n=4$} & \multicolumn{2}{|c|}{$n=6$ }  \\ \cline{1-8}
\multicolumn{2}{ |c| }{  } & \multicolumn{1}{|c|}{$KT$} & \multicolumn{1}{|c|}{$C$} &  \multicolumn{1}{|c|}{$KT$} & \multicolumn{1}{|c|}{$C$} &   \multicolumn{1}{|c|}{$KT$} & \multicolumn{1}{|c|}{$C$}   \\ \cline{1-8}
\multicolumn{2}{ |c| }{ $\omega^{10}$} &$7352$ & $7352$ & $-106244812$ & $-84923788$ & $-5089618734386048$ & $-5181970358958464$ \\ \cline{1-8}
\multicolumn{2}{ |c| }{ $\omega^{9}$} &$12100$ & $12100$ & $-40892392$ & $-98464552$ & $26333637242897408$ & $26767528167113984$ \\ \cline{1-8}
\multicolumn{2}{ |c| }{ $\omega^{8}$} &$-18868$ & $-18868$ & $135740632$ & $176373400$ & $-26132678464882128$ & $-26556943437149136$ \\ \cline{1-8}
\multicolumn{2}{ |c| }{ $\omega^{7}$} &$-16$ & $-16$ & $81031412$ & $30483572$ & $18961525460403712$ & $19282331500463872$ \\ \cline{1-8}
\multicolumn{2}{ |c| }{ $\omega^{6}$} &$-19124$ & $-19124$ & $70025564$ & $154082012$ & $-41268295304316624$ & $-41948393922548432$ \\ \cline{1-8}
\multicolumn{2}{ |c| }{ $\omega^{5}$} &$29448$ & $29448$ & $-188927128$ & $-264857368$ & $41815776250571680$ & $42495766908786848$ \\ \cline{1-8}
\multicolumn{2}{ |c| }{ $\omega^{4}$} &$-14272$ & $-14272$ & $71097428$ & $118825172$ & $-25207995553964480$ & $-25621419777084608$ \\ \cline{1-8}
\multicolumn{2}{ |c| }{ $\omega^{3}$} &$13576$ & $13576$ & $-71628932$ & $-116091140$ & $22311420427155024$ & $22676270315709264$ \\ \cline{1-8}
\multicolumn{2}{ |c| }{ $\omega^{2}$} &$-13352$ & $-13352$ & $98553148$ & $124139068$ & $-15990083236426320$ & $-16248280122238544$ \\ \cline{1-8}
\multicolumn{2}{ |c| }{ $\omega^{1}$} &$2780$ & $2780$ & $-4562444$ & $-18136844$ & $5898804809613840$ & $5996288593045520$ \\ \cline{1-8}
\multicolumn{2}{ |c| }{ $1 $} &$-5812$ & $-5812$ & $48068144$ & $56560304$ & $-5891958922292320$ & $-5986195442605152$ \\ \cline{1-8}

\end{tabular}
\caption{\normalsize{Torsions for the lift of the holonomy  with trace of the meridian $-2$, for the $n$-dimensional representations. 
 Again the table gives the coefficients with respect to a $\mathbf Q$-basis for $\mathbf{Q}(\omega)$. }}
\label{table:tra-2}
\end{footnotesize}
\end{table}

Of course, for $n=2$ and for any lift of the holonomy, the torsion of $KT$ and the torsion of $C$ is the same.
 Notice that for the 4-dimensional representation, they are also the same for one lift but different for the other,
and that they differ for both lifts when we use  the 6-dimensional representation.

As said in the introduction, when $n=2$, these had been computed   by Dunfield, Friedl and Jackson in \cite{DFJ}.
They computed numerically a twisted Alexander invariant (which are not mutation invariant) for all knots up to 15 crossings, and the torsions
computed here are just the evaluations at $\pm 1$.


\section{Other mutations and other representations}
\label{sec:concluding}

The mutation considered in this paper is called $(0,4)$-mutation, because the involved surface is planar and has 4 boundary components.
By tubing along invariant arcs of the knot, this is a particular case of the so called $(2,0)$-mutation, namely, the mutation along a closed surface of genus 2 and using the hyperelliptic involution.

 In \cite{Ruberman}
Ruberman proved that a $(2,0)$-mutant of a hyperbolic manifold is again hyperbolic. The behavior of invariants under $(2,0)$-mutation
has been investigated by many authors, see \cite{CooperLickorish,DGST,MortonRyder} for instance. Unfortunately, our arguments do not apply, as in Section~\ref{sec:trace-2}
we require two involutions, and in genus two mutation we can use only the hyperelliptic involution. So we arise: 

\begin{question}
 Is $\tor(M,\rho)$ invariant under genus two mutation?
\end{question}

The three dimensional representation of $\operatorname{SL}(2,\mathbf C)$ is conjugate to the adjoint representation in the automorphism group of the Lie algebra
$\mathfrak{sl}(2,\mathbf C)$. The representation $Ad\rho$ is not acyclic, but a natural choice of basis for homology has been given in \cite{Porti},
hence its torsion is well defined. Moreover, we have:

\begin{proposition}[\cite{Porti}]
 The torsion $\tor(M,Ad\rho)$ is invariant under 
 $(2,0)$-mutation.
\end{proposition}

The proof is straightforward, as $H^1(S; Ad\rho_0)$ is the cotangent space to the variety of characters of $S$, and the action of the hyperelliptic
involution is trivial on the variety of characters of $S$.

We have seen that  if we compose the lift of the holonomy with the 6-dimensional representation of $\operatorname{SL}(2,\mathbf C)$
(or the 4-dimensional one when the trace of the meridian is $-2$), the torsion is not invariant under $(2,0)$-mutation,
as it is not invariant under $(0,4)$-mutation, see the example of the previous section.

\begin{question}
Working with the lift of the holonomy with trace of the meridian 2, is the torsion of the 4-dimensional representation invariant under $(0,4)$-mutation?
\end{question}

To conclude, we notice that our arguments do not apply if we tensorize $\rho:\pi_1(M)\to \operatorname{SL}(2,\mathbf C)$ with the abelianization map
$\pi_1(M)  \twoheadrightarrow \mathbf Z=\langle t\mid\rangle$.
This torsion gives the twisted polynomial in $\mathbf C[t^{\pm}]$ studied in \cite{DFJ},
where it is proved not to be mutation invariant.
If we could apply the arguments of this paper, then we would be in the generic situation of Proposition~\ref {prop:generic} and Section~\ref{sec:trace-2},
 because  $\operatorname{trace}(\rho(\mu))=\pm (t+1x/t)\neq 2$.
 Notice that two of the involutions in Figure~\ref{fig:involutions} reverse the orientation of the meridian, and only one preserves them.
Thus we can use only a single involution, and at least two involutions are required in our argument from Section~\ref{sec:trace-2}, more precisely in 
Corollary~\ref{coro:tauasigmab}(ii).

\

\begin{footnotesize}

\textsc{Departament de Matem\`atiques, Universitat Aut\`onoma de Barcelona.}

\textsc{08193 Cerdanyola del Vall\`es, Spain}

{pmenal@mat.uab.cat, porti@mat.uab.cat}

\bibliographystyle{plain}



\end{footnotesize}

\end{document}